\documentclass[11pt,reqno]{amsart}
\usepackage{graphicx}
\usepackage{verbatim}
\usepackage{textcomp}
\usepackage{amssymb}
\usepackage{cite}
\usepackage{amsmath}
\usepackage{latexsym}
\usepackage{amscd}
\usepackage{amsthm}
\usepackage{mathrsfs}
\usepackage{xypic}
\usepackage{bm}
\usepackage{url}
\usepackage{hyperref}

\newtheorem{thm}{Theorem}[section]
\newtheorem{corr}[thm]{Corollary}

\theoremstyle{definition}

\theoremstyle{remark}
\newtheorem{rem}{Remark}[section]
\numberwithin{equation}{section}
\begin{document}
\title[Estimates for eigenvalues of the operator $L_r$]
{Estimates for eigenvalues of the operator $L_r$ }
\author{Guangyue Huang}
\address{College of Mathematics and Information Science, Henan Normal
University, Xinxiang, Henan 453007, People's Republic of China}
\email{hgy@henannu.edu.cn }
\author{Xuerong Qi}
\address{School of Mathematics and Statistics, Zhengzhou University, Zhengzhou, Henan 450001,
People's Republic of China} \email{xrqi@zzu.edu.cn}

\maketitle

\begin{abstract}
In this paper, we consider an eigenvalue problem of the elliptic operator
$$
L_r={\rm div}(T^r\nabla\cdot )$$ on compact submanifolds in
arbitrary codimension of space forms $\mathbb{R}^N(c)$ with
$c\geq0$. Our estimates on eigenvalues are sharp.
\end{abstract}

\section{Introduction}

Let $x: M\rightarrow \mathbb{R}^N(c)$ be an $n$-dimensional orientable closed connected submanifold of an $N$-dimensional space form $\mathbb{R}^N(c)$ of
constant sectional curvature c, where  $\mathbb{R}^N(c)$ is
Euclidean space $\mathbb{R}^N$ when $c=0$, $\mathbb{R}^N(c)$ is a
unit sphere $\mathbb{S}^N$ when $c=1$, and $\mathbb{R}^N(c)$ is a
hyperbolic space $\mathbb{H}^N$ when $c=-1$. Let $\{e_A\}_{A=1}^N$
be an orthonormal basis along $M$ such that $\{e_i\}_{i=1}^n$ are
tangent to $M$ and $\{e_\alpha\}_{\alpha=n+1}^N$ are normal to $M$.
Denote by $\{\theta_i\}_{i=1}^n$ and
$\{\theta_\alpha\}_{\alpha=n+1}^N$ the dual frame,
respectively.
%We adopt the following convention on the range of
%indices:
%$$1\leq i,j,k,\cdots\leq n;\ \ \ \  n+1\leq \alpha,\beta,\gamma,\cdots\leq N;\ \ \ \
%1\leq A,B,C,\cdots\leq N.$$
Then we have the following structure
equation (see \cite{Cao2007}):
\begin{equation}\label{1Section1}
dx=\sum_i\theta_ie_i,
\end{equation}
\begin{equation}\label{1Section2}
de_i=\sum_j\theta_{ij}e_j+\sum_{\alpha,j}h_{ij}^\alpha
\theta_je_\alpha-c\theta_ix,
\end{equation}
\begin{equation}\label{1Section3}
de_\alpha=-\sum_{i,j}h_{ij}^{\alpha}
\theta_je_i+\sum_{\beta}\theta_{\alpha\beta}e_\beta,
\end{equation}
where $h_{ij}^{\alpha}$ denote the components of the second fundamental form of $x$. Let
$B_{ij}=\sum_{\alpha=n+1}^Nh_{ij}^\alpha e_\alpha$. If
$r\in\{0,1,\cdots,n-1\}$ is even, the operator $L_r$ is defined by
\begin{equation}\label{1Section4}
L_r(f)=\sum_{i,j}T^r_{ij}f_{ij},
\end{equation} and
\begin{equation}\label{1Section5}
L_{r-1}(f)=\sum_{i,j}T^{r-1}_{\alpha ij}
f_{ij}e_\alpha.
\end{equation}
Here $T^r$ is given by
\begin{equation}\label{1Section6}
T^r_{ij}=\frac{1}{r!}\sum_{\mbox{\tiny$\begin{array}{c}
i_1\cdots i_{r} \\
j_1\cdots j_{r}\end{array}$}}\delta_{i_1\cdots i_r i}^{j_1\cdots
j_r j}\langle B_{i_1j_1},B_{i_2j_2}\rangle\cdots\langle
B_{i_{r-1}j_{r-1}},B_{i_rj_r}\rangle;
\end{equation}
\begin{equation}\label{1Section7}
T^{r-1}_{\alpha
ij}=\frac{1}{(r-1)!}\sum_{\mbox{\tiny$\begin{array}{c}
i_1\cdots i_{r-1}\\
j_1\cdots j_{r-1}\end{array}$}}\delta_{i_1\cdots i_{r-1}
i}^{j_1\cdots j_{r-1} j}\langle
B_{i_1j_1},B_{i_2j_2}\rangle\cdots\langle
B_{i_{r-3}j_{r-3}},B_{i_{r-2}j_{r-2}}\rangle
h_{i_{r-1}j_{r-1}}^{\alpha},
\end{equation}
 $\delta_{i_1\cdots i_r i}^{j_1\cdots j_r j}$ is the
generalized Kronecker symbols. It has been shown in \cite{Cao2007} that $T^{r}$ is symmetric and divergence-free. When $r$ is even,
$$ L_r(f)={\rm div}(T_r\nabla f ),$$
and the corresponding
$r$th mean curvature function $S_r$ and $(r+1)$th mean curvature
vector field $\mathbf{S}_{r+1}$ are given by
\begin{equation}\label{1Section8}\aligned
S_r=&\frac{1}{r!}\sum_{\mbox{\tiny$\begin{array}{c}
i_1\cdots i_{r}\\
j_1\cdots j_{r}\end{array}$}}\delta_{i_1\cdots i_r}^{j_1\cdots
j_r}\langle B_{i_1j_1},B_{i_2j_2}\rangle\cdots\langle
B_{i_{r-1}j_{r-1}},B_{i_rj_r}\rangle\\
=&\frac{1}{r}T^{r-1}_{\alpha ij}h_{ij}^\alpha\\
=&\binom{n}{r}H_r;
\endaligned\end{equation}
\begin{equation}\label{1Section9}\aligned
\mathbf{S}_{r+1}=&\frac{1}{(r+1)!}\sum_{\mbox{\tiny$\begin{array}{c}
i_1\cdots i_{r+1}\\
j_1\cdots j_{r+1}\end{array}$}}\delta_{i_1\cdots i_{r+1}}^{j_1\cdots
j_{r+1}}\langle B_{i_1j_1},B_{i_2j_2}\rangle\cdots\langle
B_{i_{r-1}j_{r-1}},B_{i_rj_r}\rangle B_{i_{r+1}j_{r+1}}\\
=&\frac{1}{r+1}T^{r}_{ij}h_{ij}^\alpha e_\alpha\\
=&\binom{n}{r+1}\mathbf{H}_{r+1}.
\endaligned\end{equation}
It has also been shown in \cite{Cao2007} that for any even integer $r\in\{0,1,\cdots,n-1\}$, we have
\begin{equation}\label{1Section10}
{\rm trace}(T^r)=(n-r)S_r
\end{equation} and
\begin{equation}\label{1Section11}
L_r(x)=(r+1)\mathbf{S}_{r+1}-c(n-r)S_r\,x,
\end{equation}
\begin{equation}\label{1Section12}
L_r(e_\alpha)=-\sum_{i,j,k}T^r_{ij}h_{ik,j}^\alpha
e_k-\sum_{i,j,k,\beta}T^r_{ij}h_{ik}^\alpha h_{jk}^\beta
e_\beta+c\sum_{i,j}T^r_{ij}h_{ij}^\alpha x.
\end{equation} When $M$ is
a hypersurface of a space form, we have
\begin{equation}\label{1Section13}
L_0(f)=\Delta(f),\ \ \ \
L_1(f)=\Box(f)=(nH\delta_{ij}-h_{ij})f_{ij},
\end{equation} where the operator $\Box$ was introduced by
Cheng-Yau in \cite{Cheng77} and studied by many mathematicians.
In \cite{Alencar1993}, Alencar, do Carmo and Rosenberg
generalized Reilly's inequality
to more general operators $L_r$ than the Laplacian.
That is,
they proved that when $M$ is an orientable closed hypersurface of $\mathbb{R}^{n+1}$
with $H_{r+1}>0$,
\begin{equation}\label{1Section14}
\lambda_1^{L_r}\int\limits_M H_r\,dv\leq c(r)\int\limits_M
H_{r+1}^2\,dv
\end{equation} and equality holds precisely if $M$ is a sphere.
Here $c(r)=(n-r)\binom{n}{r}$. In \cite{Grosjean2000}, Grosjean
obtained the following similar optimal upper bound for
$\lambda_1^{L_r}$ of closed hypersurfaces of any space form with
$H_{r+1}>0$ and convex isometric immersion $x$:
\begin{equation}\label{1Section15}
\lambda_1^{L_r}\,{\rm vol}(M)\leq c(r)\int\limits_M
\frac{H_{r+1}^2+cH_{r}^2}{H_r}\,dv
\end{equation} and equality holds if and only if $x(M)$ is an umbilical
sphere. For eigenvalues of $L_r$ and
some important elliptic operators, see also
\cite{Alencar2001,Alias2004,Alencar2013,
LiWang2012,Cheng2008,Cheng2014} and
references therein.

In this paper, we assume that $L_r$ is
elliptic on $M$, for some even integer $r\in\{0,1,\cdots,n-1\}$. The purpose of this
paper is to study the following closed eigenvalue problem of the elliptic operator $L_r$:
\begin{equation}\label{1Int1}
L_r(u)=-\lambda u
\end{equation} on compact submanifolds in arbitrary codimension of space forms.
 We know that the set of eigenvalues consists of
a sequence
$$0=\lambda_{0}^{L_r}<\lambda_{1}^{L_r}\leq\lambda_{2}^{L_r}\leq\cdots\leq\lambda_{k}^{L_r}
\cdots\rightarrow+\infty.$$
Denote by $u_i$ the
normalized eigenfunction corresponding to $\lambda_{i}^{L_r}$ such
that $\{u_i\}_0^\infty$ becomes an orthonormal basis of $L^2(M)$,
that is
\begin{equation*}
\left\{\begin{array}{l} L_r(u_i)=-\lambda_{i}^{L_r} u_i, \\
\int_{M}u_iu_j\,dv=\delta_{ij},\ \ {\rm for\ any}\ i, j=0,1,\cdots.
\end{array}\right.
\end{equation*}
We will prove the following results:

\begin{thm}\label{thmInt1}
Let $(M, g)$ be an $n$-dimensional orientable closed connected submanifold of a space form $\mathbb{R}^N(c)$ with $c\geq0$. Assume that $L_r$ is
elliptic on $M$, for some even integer $r\in\{0,1,\cdots,n-1\}$.
Then we have
\begin{equation}\label{1th1}
\lambda_{1}^{L_r}\int\limits_M H_r\,dv\leq c(r)\int\limits_M
(|\mathbf{H}_{r+1}|^2\,dv+c\,H_r^2)\,dv;
\end{equation}

\begin{equation}\label{1th2}
\sum\limits_{i=1}^n\sqrt{\lambda_{i}^{L_r}}\leq \frac{n}{{\rm
vol}(M)}\sqrt{(n-r)\int\limits_M S_r\,dv\int\limits_M
(|\mathbf{H}|^2+c)\,dv},
\end{equation}
where $c(r)=(n-r)\binom{n}{r}$.

In particular, for $c=0$, the equality in \eqref{1th1} holds if and
only if $M$ is a sphere in $\mathbb{R}^{n+1}$; for $c=1$,
the equality in \eqref{1th1} holds if and only if $x$ is
$r$-minimal. For $c=0$, the equality in \eqref{1th2} holds if and
only if $M$ is a sphere in $\mathbb{R}^{n+1}$; for $c=1$,
the equality in \eqref{1th2} holds if and only if $x$ is minimal.
\end{thm}

Using the fact
$$\lambda_{1}^{L_r}\leq\lambda_{2}^{L_r}\leq\cdots\leq\lambda_{n}^{L_r},$$
we have
$$\sum\limits_{i=1}^n\sqrt{\lambda_{i}^{L_r}}
\geq n\sqrt{\lambda_{1}^{L_r}}.$$ Therefore, we obtain the following
upper bound of the first eigenvalue $\lambda_{1}^{L_r}$ from
\eqref{1th2}:

\begin{corr}\label{corr1} Under the assumption of Theorem \ref{thmInt1},
we have
\begin{equation}\label{1corr1}
\lambda_{1}^{L_r}\leq \frac{n-r}{({\rm vol}(M))^2}\int\limits_M
S_r\,dv\int\limits_M (|\mathbf{H}|^2+c)\,dv.
\end{equation} In particular, for $c=0$, the equality in \eqref{1corr1} holds if and
only if $M$ is a sphere in $\mathbb{R}^{n+1}$; for $c=1$,
the equality in \eqref{1corr1} holds if and only if $x$ is minimal.
\end{corr}

In particular,
$$\sqrt{\lambda_{1}^{L_r}}\leq\sqrt{\lambda_{2}^{L_r}}\leq\cdots\leq\sqrt{\lambda_{n}^{L_r}}
<\sum\limits_{i=1}^n\sqrt{\lambda_{i}^{L_r}}.$$ Hence, we also
obtain

\begin{corr}\label{corr2} Under the assumption of Theorem
\ref{thmInt1}, we have
\begin{equation}\label{1corr2}
\lambda_{n}^{L_r}< \frac{n^2(n-r)}{({\rm vol}(M))^2}\int\limits_M
S_r\,dv\int\limits_M(|\mathbf{H}|^2+c)\,dv.
\end{equation}
\end{corr}

\begin{rem}
When $N=n+1$ and $c=0$, our estimate \eqref{1th1} becomes the result
\eqref{1Section14} of  Alencar, do Carmo and Rosenberg in
\cite{Alencar1993}. For $N=n+1$ and $c\geq0$, our estimate
\eqref{1th1} seems like the estimate \eqref{1Section15} of Grosjean
in \cite{Grosjean2000}. But our estimate \eqref{1th1} is independent
of the convex isometric immersion.

\end{rem}

\begin{rem}
Clearly, our estimate \eqref{1corr1} is new. Moreover, we obtain
estimates on high order eigenvalues of the elliptic operator $L_r$
on submanifolds of space forms with arbitrary codimension.

\end{rem}

\section{Proof of results}

In order to complete our proof, we need the following lemma:

\noindent{\bf Lemma 2.1.} {\it Under the assumption of Theorem
\ref{thmInt1},
 for any function $h_A\in C^2(M)$ satisfying
\begin{equation}\label{Lemma21}
\int\limits_M h_A u_0u_B=0, \ \ \ {\rm for}\ B=1,\cdots,A-1,
\end{equation} we have
\begin{equation}\label{Lemma23} \lambda_{A}^{L_r}\int\limits_M
\langle T^r\nabla h_A,\nabla h_A\rangle\,dv\leq \int\limits_M |{\rm
div}(T^r\nabla h_A)|^2\,dv;
\end{equation}
and
\begin{equation}\label{Lemma22}
\sqrt{\lambda_{A}^{L_r}}\int\limits_M|\nabla
h_A|^2\,dv\leq\delta\int\limits_M \langle T^r\nabla h_A,\nabla
h_A\rangle\,dv+\frac{1}{4\delta}\int\limits_M(\Delta h_A)^2\,dv,
\end{equation}
where $\delta$ is any positive constant.}

\proof We let $\varphi_A=h_A u_0-u_0\int\limits_M h_A u_0^2\,dv$.
Then
\begin{equation}\label{lemma2proof1} \int_M  \varphi_A u_0\,dv=0.
\end{equation} It has been shown  from \eqref{Lemma21} that
\begin{equation}\label{lemma2proof2}
\int\limits_M \varphi_A u_B\,dv=0, \ \ \ {\rm for}\ B=1,\cdots,A-1.
\end{equation} Hence, we have from the Rayleigh-Ritz inequality
\begin{equation}\label{lemma2proof3}
\lambda_{A}^{L_r}\int\limits_M\varphi_A
^2\,dv\leq-\int\limits_M\varphi_A L_r(\varphi_A)\,dv.
\end{equation}
Since $u_0$ is a nonzero constant satisfying $u_0^2\,{\rm
vol}(M)=1$, and $T^{r}$ is symmetric and divergence-free, a direct calculation
yields
\begin{equation}\label{lemma2proof4}\aligned
-\int\limits_M\varphi_A L_r(\varphi_A)\,dv
=&-\int\limits_M\varphi_A\,
{\rm div}(T^r\nabla (h_A u_0))\,dv\\
=&\int\limits_M\langle T^r\nabla (h_A u_0),\nabla (h_A u_0)\rangle\,dv\\
=&u_0^2\int\limits_M\langle T^r\nabla h_A ,\nabla h_A\rangle\,dv.
\endaligned
\end{equation}
Putting \eqref{lemma2proof4} into the inequality
\eqref{lemma2proof3} gives
\begin{equation}\label{lemma2proof5}
\lambda_{A}^{L_r}\int\limits_M \varphi_A^2\,dv\leq
u_0^2\int\limits_M\langle T^r\nabla h_A ,\nabla h_A\rangle\,dv.
\end{equation}

We define \begin{equation}\label{addlemma2proof1}
\omega_A:=-\int\limits_M\varphi_A\, {\rm div}(T^r\nabla (h_A
u_0))\,dv=u_0^2\int\limits_M\langle T^r\nabla h_A ,\nabla
h_A\rangle\,dv.
\end{equation} Then \eqref{lemma2proof5} gives
\begin{equation}\label{addlemma2proof2}
\lambda_{A}^{L_r}\int\limits_M \varphi_A^2\,dv\leq \omega_A.
\end{equation}
From the Schwarz inequality and \eqref{addlemma2proof2}, we obtain
\begin{equation}\label{addlemma2proof3}\aligned
\lambda_{A}^{L_r}\omega_A^2=&\lambda_{A}^{L_r}\left(\int\limits_M\varphi_A\,
{\rm div}(T^r\nabla (h_A u_0))\,dv\right)^2\\
\leq&\lambda_{A}^{L_r}\left(\int\limits_M
\varphi_A^2\,dv\right)\left(\int\limits_M |{\rm div}(T^r\nabla (h_A
u_0))|^2\,dv\right)\\
\leq&\omega_A \int\limits_M |{\rm div}(T^r\nabla (h_A u_0))|^2\,dv,
\endaligned\end{equation} which gives
\begin{equation}\label{addlemma2proof4}
\lambda_{A}^{L_r}\omega_A\leq\int\limits_M |{\rm div}(T^r\nabla (h_A
u_0))|^2\,dv.
\end{equation} Combining \eqref{addlemma2proof1} with
\eqref{addlemma2proof4} yields the inequality \eqref{Lemma23}.

On the other hand, from the the Stokes formula, one gets
$$\aligned -u_0\int\limits_M\varphi_A\Delta h_A \,dv=&
-\int\limits_M \varphi_A\Delta(h_A u_0)\,dv\\
 =&-\int\limits_M
\Big(h_A u_0-u_0\int\limits_Mh_A
u_0^2\,dv\Big)\Delta(h_A u_0)\,dv \\
=&\int\limits_M|\nabla (h_A u_0)|^2\,dv\\
=&u_0^2\int\limits_M|\nabla h_A|^2\,dv.
\endaligned$$
Therefore, for any positive constant $\delta$, we derive from
\eqref{lemma2proof5}
\begin{equation}\label{lemma2proof6} \aligned
\sqrt{\lambda_{A}^{L_r}}u_0^2\int\limits_M|\nabla
h_A|^2\,dv=&-\sqrt{\lambda_{A}^{L_r}}u_0\int\limits_M
\varphi_A\Delta h_A \,dv\\
\leq&\delta\lambda_{A}^{L_r}\int\limits_M
\varphi_A^2\,dv+\frac{1}{4\delta}u_0^2\int\limits_M(\Delta
h_A)^2\,dv\\
\leq&\delta u_0^2\int\limits_M\langle T^r\nabla h_A,\nabla
h_A\rangle\,dv+\frac{1}{4\delta}u_0^2\int\limits_M(\Delta
h_A)^2\,dv.
\endaligned
\end{equation} The desired inequality \eqref{Lemma22} is
obtained.\endproof

\begin{proof}[Proof of the estimate \eqref{1th1} in Theorem \ref{thmInt1}]

For $c=0$, according to the orthogonalization of Gram and Schmidt,
we get that there exists an orthogonal matrix $O=(O_A^B)$ such that
\begin{equation}\label{thproof1}
\sum_{C=1}^N\int\limits_M O_A^C x_Cu_0u_B
=\sum_{\gamma=1}^NO_A^C\int\limits_M x_C u_0u_B=0, \ \ \ {\rm for}\
B=1,\ldots,A-1.
\end{equation} Taking $h_A=\sum_{C=1}^NO_A^C
x_C$ in \eqref{Lemma23}, and summing over $A$ from 1 to $N$, we
obtain
\begin{equation}\label{thproof2} \sum_{A=1}^N\lambda_{A}^{L_r}\int\limits_M \langle
T^r\nabla h_A,\nabla h_A\rangle\,dv\leq \sum_{A=1}^N\int\limits_M
|{\rm div}(T^r\nabla h_A)|^2\,dv.
\end{equation}
Since $L_r$ is elliptic, namely $T^{r}$ is positive definite,
we have
\begin{equation}\label{thproof3}\sum_{A=1}^N\lambda_{A}^{L_r}\int\limits_M
\langle T^r\nabla h_A,\nabla
h_A\rangle\,dv\geq\lambda_{1}^{L_r}\sum_{A=1}^N\int\limits_M \langle
T^r\nabla h_A,\nabla h_A\rangle\,dv.
\end{equation} Therefore, from \eqref{thproof2} and the orthogonal matrix $O$, we derive

\begin{equation}\label{thproof4} \lambda_{1}^{L_r}\sum_{A=1}^N\int\limits_M \langle
T^r\nabla x_A,\nabla x_A\rangle\,dv\leq \sum_{A=1}^N\int\limits_M
|{\rm div}(T^r\nabla x_A)|^2\,dv.
\end{equation}
Let $E_1,\cdots,E_N$ be a canonical orthonormal basis of
$\mathbb{R}^N$, then $x_A=\langle E_A,x\rangle$ and
\begin{equation}\label{addthproof4}\nabla (x_A)=\langle
E_A,e_i\rangle e_i=E_A^{\top},\end{equation} where $\top$ denote the
tangent projection to $M$. Therefore,
$$\aligned\sum_{A=1}^N\langle
T^r\nabla x_A,\nabla x_A\rangle=&\sum_{A=1}^NT^r_{ij}\langle
E_A,e_i\rangle\langle E_A,e_j\rangle\\
=&T^r_{ij}\langle e_i,e_j\rangle\\
=&{\rm trace}(T^r)\\
=&(n-r)S_r,\endaligned$$ which shows that $S_r>0$ since $T^{r}$ is positive definite. Using the definition of $L_r$, we have
$$\sum_{A=1}^N|{\rm div}(T^r\nabla x_A)|^2=\sum_{A=1}^N\langle E_A,L_r(x)\rangle^2=|L_r(x)|^2=(r+1)^2|\mathbf{S}_{r+1}|^2.$$
Thus, we derive from \eqref{thproof4}
\begin{equation}\label{thproof5}
(n-r)\lambda_{1}^{L_r}\int\limits_M
S_r\,dv\leq(r+1)^2\int\limits_M|\mathbf{S}_{r+1}|^2\,dv.
\end{equation}
By virtue of the relationships between $S_r$ and $H_r$, we deduce to
\begin{equation}\label{thproof6}
\lambda_{1}^{L_r}\int\limits_M H_r\,dv\leq c(r)\int\limits_M
|\mathbf{H}_{r+1}|^2\,dv.
\end{equation}

When $c=1$,
$$\mathbb{S}^N=\{x\in \mathbb{R}^{N+1};\ |x|^2=x_0^2+x_1^2+\cdots x_N^2=\frac{1}{c} \}.$$ Using the
similar method, we can derive
\begin{equation}\label{thproof7} \lambda_{1}^{L_r}\sum_{A=0}^{N
}\int\limits_M\langle T^r\nabla x_A,\nabla x_A\rangle\,dv\leq
\sum_{A=0}^{N}\int\limits_M|{\rm div}(T^r\nabla x_A)|^2\,dv.
\end{equation}
Putting $$\sum_{A=0}^{N}\langle T^r\nabla x_A,\nabla x_A\rangle={\rm
trace}(T^r)=(n-r)S_r$$ and
$$\aligned
\sum_{A=0}^{N}|{\rm div}(T^r\nabla
x_A)|^2=&|L_r(x)|^2\\
=&(r+1)^2|\mathbf{S}_{r+1}|^2+c^2(n-r)^2S_r^2|x|^2\\
=&(r+1)^2|\mathbf{S}_{r+1}|^2+c(n-r)^2S_r^2
\endaligned$$ into \eqref{thproof7} gives
\begin{equation}\label{thproof8} (n-r)\lambda_{1}^{L_r}\int\limits_M S_r\,dv\leq
\int\limits_M [(r+1)^2|\mathbf{S}_{r+1}|^2+c(n-r)^2S_r^2]\,dv.
\end{equation} Hence, the desired estimate \eqref{1th1} is derived.

Next, we consider the case that equalities occur. If $c\geq0$ and
the equality in \eqref{1th1} holds, then inequalities
\eqref{lemma2proof3}, \eqref{addlemma2proof3} and \eqref{thproof3}
become equalities. Hence, we have
\begin{equation}\label{thproof9}
\lambda_{1}^{L_r}=\lambda_{2}^{L_r}=\cdots=\lambda_{N}^{L_r}=\mu;
\end{equation}
\begin{equation}\label{thproof10}
L_r(\varphi_A)=-\mu\,\varphi_A,
\end{equation} where $\mu$ is a constant. When $c=0$, from \eqref{thproof10} and
\eqref{1Section11}, we can infer that the vector field
$\varphi=(\varphi_1,\cdots,\varphi_N)$ is parallel with
$\mathbf{S}_{r+1}$. Thus, we obtain
\begin{equation}\label{thproof11}
\frac{1}{2}(|\varphi|^2)_{,i}=\langle e_i, \varphi\rangle=0,
\end{equation} which shows that $|\varphi|^2$ is constant. Hence
$M$ is a sphere in $\mathbb{R}^{n+1}$.
When $c=1$ and the equality in \eqref{1th1} holds, it is easy to see
that $\mathbf{S}_{r+1}=0$ by combining \eqref{thproof10} with
\eqref{1Section11}. That is to say that $x$ is $r$-minimal.

\end{proof}

\begin{proof}[Proof of the estimate \eqref{1th2} in Theorem \ref{thmInt1}]
For $c=0$, we taking $h_A=\sum_{C=1}^NO_A^C x_C$ in
\eqref{Lemma22}, where the matric $O$ is given by \eqref{thproof1}.
 From \eqref{addthproof4}, we get
\begin{equation}\label{2thproof1}
|\nabla h_A|^2=|\nabla x_A|^2=|E_A^{\top}|^2\leq|E_A|^2=1, \ \ \ \
\forall \ A,
\end{equation} and
\begin{equation}\label{2thproof2}
\sum_{A=1}^N|\nabla h_A|^2=\sum_{A=1}^N|\nabla x_A|^2=n.
\end{equation}
Thus, we infer
\begin{equation}\label{2thproof3}
\aligned &\sum_{A=1}^N\sqrt{\lambda_{A}^{L_r}}|\nabla
   h_{A}|^2\\
   &\geq\sum_{i=1}^n\sqrt{\lambda_{i}^{L_r}}|\nabla
   h_{i}|^2+\sqrt{\lambda_{n+1}^{L_r}}\sum\limits_{\alpha=n+1}^N|\nabla h_{\alpha}|^2\\
   &=\sum_{i=1}^n\sqrt{\lambda_{i}^{L_r}}|\nabla
   h_{i}|^2+\sqrt{\lambda_{n+1}^{L_r}}\left(n-\sum\limits_{j=1}^n|\nabla
   h_{j}|^2\right)\\
   &=\sum_{i=1}^n\sqrt{\lambda_{i}^{L_r}}|\nabla
   h_{i}|^2+\sqrt{\lambda_{n+1}^{L_r}}\sum\limits_{j=1}^n(1-|\nabla
   h_{j}|^2)\\
   &\geq\sum_{i=1}^n\sqrt{\lambda_{i}^{L_r}}|\nabla
   h_{i}|^2+\sum\limits_{j=1}^n\sqrt{\lambda_{j}^{L_r}}(1-|\nabla
   h_{j}|^2)\\
   &=\sum\limits_{i=1}^n\sqrt{\lambda_{i}^{L_r}}.
\endaligned
\end{equation}
From \eqref{1Section11}, we have $\sum_{A=1}^N(\Delta
h_{A})^2=|\Delta(x)|^2=|\mathbf{S}_1|^{2}=n^{2}|\mathbf{H}|^{2}$. Hence, taking sum
on $A$ from 1 to $N$ for \eqref{Lemma22}, we have
\begin{equation}\label{2thproof4}
\sum\limits_{i=1}^n\sqrt{\lambda_{i}^{L_r}}{\rm
vol}(M)\leq\delta(n-r)\int\limits_M S_r\,dv
+\frac{n^2}{4\delta}\int\limits_M |\mathbf{H}|^2\,dv.
\end{equation}
Minimizing the right hand side of \eqref{2thproof4} by taking
$$\delta=\frac{n}{2}\sqrt{\frac{\int_M
|\mathbf{H}|^2\,dv}{(n-r)\int_M S_r\,dv}}$$ yields
\begin{equation}\label{2thproof5}
\sum\limits_{i=1}^n\sqrt{\lambda_{i}^{L_r}}{\rm vol}(M)\leq
n\sqrt{(n-r)\int\limits_M S_r\,dv\int\limits_M |\mathbf{H}|^2\,dv}.
\end{equation}

When $c=1$, the proof is similar. We omit it here.

When $c\geq0$ and the equality in \eqref{1th2} holds, then
inequalities \eqref{lemma2proof3}, \eqref{lemma2proof6} and
\eqref{2thproof3} become equalities. Hence, we have
\begin{equation}\label{2thproof9}
\lambda_{1}^{L_r}=\lambda_{2}^{L_r}=\cdots=\lambda_{N}^{L_r}=\mu;
\end{equation}
\begin{equation}\label{2thproof10}
\Delta(\varphi_A)=-\mu\varphi_A.
\end{equation}
Similarly, we infer that, for $c=0$, the equality in \eqref{1th2}
holds if and only if $M$ is a sphere in
$\mathbb{R}^{n+1}$; for $c=1$, the equality in \eqref{1th2} holds if
and only if $x$ is minimal.

\end{proof}

\noindent{\bf Acknowledgements.} The first author's research was
supported by NSFC No. 11371018, 11171091. The second author's
research was supported by NSFC No. 11401537.

%We thank the referee for helpful suggestions which make the paper
%more readable.

\bibliographystyle{Plain}

\end{document}